\documentclass{birkjour}
\usepackage[utf8]{inputenc}
\usepackage{amssymb, mathtools, hyperref, xcolor, mathrsfs}

\usepackage[backend=biber]{biblatex}
\numberwithin{equation}{section}
\newtheorem{thm}{Theorem}[section]
\newtheorem{prop}[thm]{Proposition}
\newtheorem{lem}[thm]{Lemma}

\theoremstyle{definition}

\newcommand{\nat}{\mathbb{N}}
\newcommand{\real}{\mathbb{R}}

\newcommand{\Pee}{\mathcal{P}}
\newcommand{\NP}{\mathscr{P}}

\newcommand{\I}{\widetilde{\mathcal{I}}}
\newcommand{\K}{\mathcal{K}}

\newcommand{\dif}{\mathrm{d}}

\DeclarePairedDelimiter{\abs}{\lvert}{\rvert}
\DeclarePairedDelimiter{\norm}{\lVert}{\rVert}
\DeclarePairedDelimiter{\parens}{(}{)}
\DeclarePairedDelimiter{\set}{\{}{\}}

\DeclarePairedDelimiter{\angles}{\langle}{\rangle}

\DeclarePairedDelimiter{\coi}{\lbrack}{\lbrack}
\DeclarePairedDelimiter{\oci}{\rbrack}{\rbrack}
\DeclarePairedDelimiter{\ooi}{\rbrack}{\lbrack}

\title[Asymptotic profile of least energy solutions to the nonlinear SBP system]{\sloppy Asymptotic profile of least energy solutions to the nonlinear Schrödinger--Bopp--Podolsky system}
\begin{document}

\author{Gustavo de Paula Ramos}
\address{Instituto de Matemática e Estatística, Universidade de São Paulo, Rua do Matão, 1010, 05508-090 São Paulo SP, Brazil}
\email{gpramos@ime.usp.br}
\urladdr{http://gpramos.com}

\begin{abstract}
Consider the following nonlinear Schrödinger--Bopp--Podolsky system in $\real^3$:
\[
\begin{cases}
- \Delta v + v + \phi v = v \abs{v}^{p - 2};
\\
\beta^2 \Delta^2 \phi - \Delta \phi = 4 \pi v^2,
\end{cases}
\]
where $\beta > 0$ and $3 < p < 6$, the unknowns being
$v$, $\phi \colon \real^3 \to \real$.
We prove that, as
$\beta \to 0$ and up to translations and subsequences, least energy solutions to this system converge to a least energy solution to the following nonlinear Schrödinger--Poisson system in
$\real^3$:
\[
\begin{cases}
- \Delta v + v + \phi v = v \abs{v}^{p - 2};
\\
- \Delta \phi = 4 \pi v^2.
\end{cases}
\]

\smallskip
\sloppy \noindent \textbf{Keywords.} 
Schrödinger--Bopp--Podolsky system, Schrödinger--Poisson system, Nonlocal semilinear elliptic problem, Variational methods, Ground state, Nehari--Pohožaev manifold, Concentration-compactness.
\end{abstract}

\date{\today}
\maketitle

\section{Introduction}

We are interested in the asymptotic profile of solutions to the following \emph{nonlinear Schrödinger--Bopp--Podolsky (SBP) system} in
$\real^3$ as $\beta \to 0^+$:
\begin{equation}
\label{eqn:nonreduced_SBP}
\begin{cases}
- \Delta v + v + \phi v = v \abs{v}^{p - 2};
\\
\beta^2 \Delta^2 \phi - \Delta \phi = 4 \pi v^2,
\end{cases}
\end{equation}
where $3 < p < 6$ and we want to solve for
$v, \phi \colon \real^3 \to \real$.

The nonlinear SBP system was introduced in the mathematical literature a few years ago by d'Avenia \& Siciliano in \cite{daveniaNonlinearSchrodingerEquation2019}, where they established existence/nonexistence results of solutions to the following system in $\real^3$ in function of the parameters $p$, $q \in \real$:
\begin{equation}
\label{eqn:SBP_dS}
\begin{cases}
- \Delta v + \omega v + q^2 \phi v = v \abs{v}^{p - 2};
\\
\beta^2 \Delta^2 \phi - \Delta \phi = 4 \pi v^2,
\end{cases}
\end{equation}
where $\beta$, $\omega > 0$. As for the physical meaning of this system, if $v$, $\phi \colon \real^3 \to \real$ solve \eqref{eqn:SBP_dS}, then $v$ describes the spatial profile of a standing wave
\[\psi \parens{x, t} := e^{\mathrm{i} \omega t} v \parens{x}\]
that solves the system obtained by the minimal coupling of the Nonlinear Schrödinger Equation with the Bopp--Podolsky electromagnetic theory and $\phi$ denotes the ensuing electric potential (for more details, see \cite[Section 2]{daveniaNonlinearSchrodingerEquation2019}). Since then, there has been an increasing number of studies about systems related to \eqref{eqn:SBP_dS}. For instance,
\cite{chenCriticalSchrodingerBopp2020, chenGroundStateSolutions2022, jiaGroundStatesSolutions2022, liGroundStateSolutions2020, liuExistenceAsymptoticBehaviour2022, xiaoExistenceGroundState2023, zhuSchrodingerBoppPodolsky2021}
addressed the existence of least energy solutions,
\cite{depaularamosExistenceLimitBehavior2023, liCriticalSchrodingerBopp2023, liNormalizedSolutionsSobolev2023} considered the mass-constrained problem,
\cite{figueiredoMultipleSolutionsSchrodinger2023a, huExistenceLeastEnergySignChanging2023a, lixiongwangExistenceMultiplicitySignchanging2022, zhangSignchangingSolutionsClass2024, zhangSignchangingSolutionsSchrodinger2022}
obtained sign-changing solutions and 
\cite{damianCriticalSchrodingerBopp2024, depaularamosConcentratedSolutionsSchrodinger2024}
considered semiclassical states.

As for the asymptotic behavior as $\beta \to 0$, it is already known that solutions to a number of problems related to \eqref{eqn:nonreduced_SBP} converge to solutions of the respective system obtained by formally considering
$\beta = 0$. For instance, \cite[Theorem 1.3]{daveniaNonlinearSchrodingerEquation2019} proved such a result for radial solutions; \cite[Theorem D]{depaularamosExistenceLimitBehavior2023} extended this conclusion for least energy solutions to the mass-constrained system for
$2 < p < 14 / 5$ and a sufficiently small mass $\rho$ (notice that these solutions are also radial due to \cite[Theorem C]{depaularamosExistenceLimitBehavior2023}); \cite[Theorem 1.3]{sorianohernandezExistenceAsymptoticBehavior2023} showed that solutions to the associated eigenvalue problem in a bounded smooth domain also have such an asymptotic profile and, more recently, \cite[Theorem 1.7]{damianCriticalSchrodingerBopp2024} verified such a behavior for the critical nonlinear SBP system in the semiclassical regime under the effect of an external effective potential
$V \colon \real^3 \to \coi{0, \infty}$ which vanishes at a point
$x_0 \in \real^3$.

Before explaining our contribution, let us introduce the necessary variational framework. The function
$\K_\beta \colon \real^3 \setminus \set{0} \to \ooi{0, 1 / \beta}$
defined as
\[
\K_\beta \parens{x}
:=
\frac{1}{\abs{x}}
\parens*{1 - e^{- \abs{x} / \beta}}
\]
is a fundamental solution to
$\parens{4 \pi}^{- 1} \parens{\beta^2 \Delta^2 - \Delta}$, so
$u^2 \ast \K_\beta$ solves
\[
\beta^2 \Delta^2 \phi - \Delta \phi = 4 \pi u^2
\]
in the sense of distributions. As such, we are lead to consider the following \emph{nonlinear SBP equation} in $\real^3$:
\begin{equation}
\label{eqn:SBP}
- \Delta v + v + \parens{v^2 \ast \K_\beta} v
=
v \abs{v}^{p - 2}.
\end{equation}
We say that $v$ is a \emph{least energy solution} to \eqref{eqn:SBP} when it solves the minimization problem
\[
\I_\beta \parens{u}
=
\inf \set*{
	\I_\beta \parens{v} :
	v \in H^1 \setminus \set{0}
	\quad \text{and} \quad
	\I_\beta' \parens{v} = 0
};
\quad
u \in H^1,
\]
where the \emph{energy functional} $\I_\beta \colon H^1 \to \real$ is defined as
\[
\I_\beta \parens{v}
=
\frac{1}{2} \norm{v}_{H^1}^2
+
\frac{1}{4}
\int
	\parens{v^2 \ast \K_\beta} \parens{x} v \parens{x}^2
\dif x
-
\frac{1}{p} \norm{v}_{L^p}^p.
\]
For a proof that $\I_\beta$ is a well-defined functional of class $C^1$ and a rigorous discussion about the relationship between \eqref{eqn:nonreduced_SBP} and \eqref{eqn:SBP}, we refer the reader to \cite[Section 3.2]{daveniaNonlinearSchrodingerEquation2019}.

Given $x \in \real^3 \setminus \set{0}$,
$\K_\beta \parens{x} \to 1 / \abs{x}$
as $\beta \to 0$, so the formal limit equation obtained from \eqref{eqn:SBP} is the \emph{nonlinear Schrödinger--Poisson equation}
\begin{equation}
\label{intro:eqn:reduced_SP}
- \Delta v + v + \parens*{v^2 \ast \abs{\cdot}^{- 1}} v
=
v \abs{v}^{p - 2}.
\end{equation}
We similarly introduce a notion of least energy solution to \eqref{intro:eqn:reduced_SP} by considering the energy functional
$\I_0 \colon H^1 \to \real$ given by
\[
\I_0 \parens{v}
:=
\frac{1}{2} \norm{v}_{H^1}^2
+
\frac{1}{4}
\int \int
	\frac{v \parens{x}^2 v \parens{y}^2}{\abs{x - y}}
\dif x \dif y
-
\frac{1}{p} \norm{v}_{L^p}^p.
\]	
In this context, our main result is that, up to translations and subsequences, least energy solutions to \eqref{eqn:SBP} converge to a least energy solution to \eqref{intro:eqn:reduced_SP} as
$\beta \to 0$ when $3 < p < 6$.
\begin{thm}
\label{thm:asymptotic_profile}
Suppose that $3 < p < 6$ and given $\beta > 0$, $v_\beta$ is a least energy solution to \eqref{eqn:SBP}. It follows that if
$\beta_n \to 0$ as $n \to \infty$, then \eqref{intro:eqn:reduced_SP} has a least energy solution $\overline{v}_0$ and there exists
$\set{\xi_n}_{n \in \nat} \subset \real^3$ such that, up to subsequence,
$
\lim_{n \to \infty}
\norm{v_{\beta_n} \parens{\cdot + \xi_n} - \overline{v}_0}_{H^1}
=
0
$.
\end{thm}

We prove the theorem by arguing as in Liu \& Moroz' \cite{liuAsymptoticProfileGround2022}, where they characterized the asymptotic profile of least energy solutions to the following Schrödinger--Poisson equation in $\real^3$ as $\lambda \to \infty$:
\[
-
\Delta v
+
v
+
\frac{\lambda}{4 \pi} \parens*{v^2 \ast \abs{\cdot}^{- 1}} v
=
v \abs{v}^{p - 2},
\]
where $3 < p < 6$. Let us summarize the strategy of the proof. It is already known that when $3 < p < 6$, least energy solutions to \eqref{eqn:SBP} and \eqref{intro:eqn:reduced_SP} are minimizers of the respective energy functionals in the associated Nehari--Pohožaev manifolds (see \cite{chenGroundStateSolutions2022} for the SBP system and \cite{azzolliniGroundStateSolutions2008, ruizSchrodingerPoissonEquation2006} for the Schrödinger--Poisson system). As such, the core of the proof consists in comparing the least energy level achieved on these manifolds as $\beta \to 0$.

Let us finish the introduction with a comment on the organization of the text. In Section \ref{prelim}, (i) we recap relevant results present in the literature; (ii) we precisely define the Nehari--Pohožaev manifold and (iii) we recall its properties which we will use. Finally, we prove Theorem \ref{thm:asymptotic_profile} in Section \ref{asymptotic}.

\subsection*{Notation}

Unless denoted otherwise, functional spaces contain real-valued functions defined a.e. in $\real^3$. Likewise, we integrate over $\real^3$ whenever the domain of integration is omitted. We define $D^{1, 2}$ as the Hilbert space obtained as completion of $C_c^\infty$ with respect to the inner product
$
\angles{u, v}_{D^{1, 2}}
:=
\int \nabla u \parens{x} \cdot \nabla v \parens{x} \dif x
$.
In the following sections, we always consider a fixed
$p \in \ooi{3, 6}$.

\section{Preliminaries}
\label{prelim}

We begin by recalling the following Brézis--Lieb-type splitting property (see \cite[Lemma 2.2 (i)]{zhaoExistenceSolutionsSchrodinger2008} or \cite[Proposition 4.7]{mercuriGroundstatesRadialSolutions2016}).
\begin{lem}
\label{lem:BL}
If $w_n \rightharpoonup \overline{v}_0$ in $H^1$ and
$w_n \to \overline{v}_0$ a.e. as $n \to \infty$, then
\begin{multline*}
\int \int
	\frac{
		w_n \parens{x}^2 w_n \parens{y}^2
	}{\abs{x - y}}
\dif x \dif y
-
\\
-
\int \int
	\frac{
		\parens*{w_n \parens{x} - \overline{v}_0 \parens{x}}^2
		\parens*{w_n \parens{y} - \overline{v}_0 \parens{y}}^2
	}{\abs{x - y}}
\dif x \dif y
\xrightarrow[n \to \infty]{}
\\
\xrightarrow[n \to \infty]{}
\int \int
	\frac{
		\overline{v}_0 \parens{x}^2 \overline{v}_0 \parens{y}^2
	}{\abs{x - y}}
\dif x \dif y.
\end{multline*}
\end{lem}

The Pohožaev-type identities in the sequence were proved in \cite[Appendix A.3]{daveniaNonlinearSchrodingerEquation2019} and \cite[Theorem 2.2]{ruizSchrodingerPoissonEquation2006}.

\begin{prop}
\begin{enumerate}
\item
If $v \in H^1$ is a weak solution to \eqref{eqn:SBP}, then
\begin{multline}
\label{eqn:Pohozaev:SBP}
\frac{1}{2} \norm{v}_{D^{1, 2}}^2
+
\frac{3}{2} \norm{v}_{L^2}^2
+
\frac{5}{4}
\int \int
	\K_\beta \parens{x - y} v \parens{x}^2 v \parens{y}^2
\dif x \dif y
+
\\
+
\frac{1}{4 \beta}
\int \int
	e^{- \abs{x - y} / \beta} v \parens{x}^2 v \parens{y}^2
\dif x \dif y
-
\frac{3}{p} \norm{v}_{L^p}^p
=
0.
\end{multline}
\item
If $v \in H^1$ is a weak solution to \eqref{intro:eqn:reduced_SP}, then
\[
\frac{1}{2} \norm{v}_{D^{1, 2}}^2
+
\frac{3}{2} \norm{v}_{L^2}^2
+
\frac{5}{4}
\int \int
	\frac{v \parens{x}^2 v \parens{y}^2}{\abs{x - y}}
\dif x \dif y
-
\frac{3}{p} \norm{v}_{L^p}^p
=
0.
\]
\end{enumerate}
\end{prop}

Let $\Pee_\beta \colon H^1 \to \real$ be defined as
\begin{multline*}
\Pee_\beta \parens{v}
=
\frac{3}{2} \norm{v}_{D^{1, 2}}^2
+
\frac{1}{2} \norm{v}_{L^2}^2
+
\frac{3}{4}
\int \int
	\K_\beta \parens{x - y} v \parens{x}^2 v \parens{y}^2
\dif x \dif y
-
\\
-
\frac{1}{4 \beta}
\int \int
	e^{- \abs{x - y} / \beta}
	u \parens{x}^2 u \parens{y}^2
\dif x \dif y
-
\frac{2 p - 3}{p} \norm{u}_{L^p}^p.
\end{multline*}
To motivate the definition of $\Pee_\beta$, notice that every critical point of $\I_\beta$ is an element of the Nehari--Pohožaev manifold
\[
\NP_\beta
:=
\set*{v \in H^1 \setminus \set{0} : \Pee_\beta \parens{v} = 0}.
\]
Indeed: if $\I_\beta' \parens{v} = 0$, then both the Nehari identity
\[
\norm{v}_{H^1}^2
+
\int
	\parens{v^2 \ast \K_\beta} \parens{x} v \parens{x}^2
\dif x
-
\norm{v}_{L^p}^p
=
0
\]
and the Pohožaev-type identity \eqref{eqn:Pohozaev:SBP} hold, so
$\Pee_\beta \parens{v} = 0$.

On one hand, it seems to be unknown whether $\NP_\beta$ is a \emph{natural constraint} of $\I_\beta$ in the sense that if $v$ is a critical point of $\I_\beta|_{\NP_\beta}$, then $\I_\beta' \parens{v} = 0$. On the other hand, under more general assumptions, Chen, Li, Rădulescu \& Tang proved in \cite[Lemma 3.14]{chenGroundStateSolutions2022} that if $v$ solves the minimization problem
\[
\I_\beta \parens{v}
=
m_\beta
:=
\inf_{u \in \NP_\beta} \I_\beta \parens{u};
\quad
v \in \NP_\beta,
\]
then $v$ is a least energy solution to \eqref{eqn:SBP}. Moreover, it follows from \cite[Corollary 1.6]{chenGroundStateSolutions2022} that $m_\beta$ is actually achieved and $m_\beta > 0$. As such, we will henceforth let $v_\beta$ denote any least energy solution to \eqref{eqn:SBP}.

Suppose that $v \in H^1 \setminus \set{0}$. There exists a unique
$\tau > 0$ such that $\tau^2 v \parens{\tau \cdot} \in \NP_\beta$, which is obtained as the unique critical point of the mapping
\begin{multline*}
\ooi{0, \infty} \ni t \mapsto
\I_\beta \parens*{t^2 v \parens{t \cdot}}
=
\\
=
\frac{t^3}{2} \norm{v}_{D^{1, 2}}^2
+
\frac{t}{2} \norm{v}_{L^2}^2
+
\frac{t^3}{4}
\int
	\parens{v^2 \ast \K_{t \beta}} \parens{x} v \parens{x}^2
\dif x
-
\frac{t^{2 p - 3}}{p} \norm{v}_{L^p}^p.
\end{multline*}
Furthermore, $\Pee_\beta \parens{t^2 v \parens{t \cdot}} > 0$ for
$0 < t < \tau$ and $\Pee_\beta \parens{t^2 v \parens{t \cdot}} < 0$ for $t > \tau$.

We let $\Pee_0 \colon H^1 \to \real$ be given by
\[
\Pee_0 \parens{v}
=
\frac{3}{2} \norm{v}_{D^{1, 2}}^2
+
\frac{1}{2} \norm{v}_{L^2}^2
+
\frac{3}{4}
\int \int
	\frac{v \parens{x}^2 v \parens{y}^2}{\abs{x - y}}
\dif x \dif y
-
\frac{2 p - 3}{p} \norm{v}_{L^p}^p
\]
and we analogously define $\NP_0$, $m_0$. As before, we can associate each $v \in H^1 \setminus \set{0}$ to a unique $\tau > 0$ such that
$\tau^2 v \parens{\tau \cdot} \in \NP_0$. It follows from Azzollini \& Pomponio's \cite[Theorem 1.1]{azzolliniGroundStateSolutions2008} that $m_0 > 0$ and \eqref{intro:eqn:reduced_SP} has a least energy solution obtained as a minimizer of $\I_0|_{\NP_0}$, so we will henceforth let $v_0$ denote any of these solutions. Let us recall a couple of properties of $\NP_0$ that follow directly from \cite[Lemma 2.3]{azzolliniGroundStateSolutions2008} and which will be important for us.

\begin{lem}
\label{lem:NP_0-is-natural-constraint}
\begin{enumerate}
\item
The Nehari--Pohožaev manifold $\NP_0$ is a natural constraint of
$\I_0$.
\item
$\inf_{v \in \NP_0} \norm{v}_{L^p} > 0$.
\end{enumerate}
\end{lem}

We will also use the fact that minimizing sequences of
$\I_0|_{\NP_0}$ induce a sequence of measures which falls on the compactness case in P.--L. Lions' \cite[Lemma I.1]{lionsConcentrationcompactnessPrincipleCalculus1984}.

\begin{lem}
[{\cite[Lemma 2.6]{azzolliniGroundStateSolutions2008}}]
\label{lem:compactness}
Suppose that $\parens{u_n}_{n \in \nat}$ is a minimizing sequence of
$\I_0|_{\NP_0}$ and given $n \in \nat$, $\mu_n$ denotes the measure which takes each Lebesgue-measurable set $\Omega$ to
\begin{multline*}
\mu_n \parens{\Omega}
:=
\\
:=
\int_\Omega
\frac{p - 3}{2 p - 3} \abs*{\nabla u_n \parens{x}}^2
+
\frac{p - 2}{2 p - 3} u_n \parens{x}^2
+
\frac{p - 2}{2 \parens{2 p - 3}}
\int
	\frac{u_n \parens{x}^2 u_n \parens{y}^2}{\abs{x - y}}
\dif y
\dif x.
\end{multline*}
It follows that there exists
$\set{\xi_n}_{n \in \nat} \subset \real^3$
for which we can associate each $\delta > 0$ with an $r_\delta > 0$ such that
$\mu_n \parens{B_{r_\delta} \parens{\xi_n}} \geq m_0 - \delta$
for every $n \in \nat$.
\end{lem}

\section{Asymptotic profile of least energy solutions to \eqref{eqn:SBP}}
\label{asymptotic}

Let us develop the preliminary results needed to prove the theorem. We begin by obtaining an upper bound for $\limsup_{\beta \to 0} m_\beta$.

\begin{lem}
\label{lem:m0_geq_mbeta}
$\limsup_{\beta \to 0} m_\beta \leq m_0$.
\end{lem}
\begin{proof}
$v_0 \in \NP_0$, so
\[
\Pee_\beta \parens{v_0}
=
-
\frac{1}{4}
\int \int
	\parens*{
		\frac{3}{\abs{x - y}}
		+
		\frac{1}{\beta}	
	}
	e^{- \abs{x - y} / \beta} v_0 \parens{x}^2 v_0 \parens{y}^2
\dif x \dif y
<
0,
\]
so there exists a unique $\overline{t}_\beta \in \ooi{0, 1}$ such that
$\overline{t}_\beta^2 v_0 \parens{\overline{t}_\beta \cdot} \in \NP_\beta$, i.e.,
\begin{multline*}
\frac{3}{2} \overline{t}_\beta^3 \norm{v_0}_{D^{1, 2}}^2
+
\frac{1}{2} \overline{t}_\beta \norm{v_0}_{L^2}^2
+
\frac{3}{4} \overline{t}_\beta^3
\int \int
	\K_{\overline{t}_\beta \beta} \parens{x - y} v_0 \parens{x}^2 v_0 \parens{y}^2
\dif x \dif y
-
\\
-
\frac{\overline{t}_\beta^2}{4 \beta}
\int \int
	e^{- \abs{x - y} / \parens{\overline{t}_\beta \beta}}
	v_0 \parens{x}^2 v_0 \parens{y}^2
\dif x \dif y
=
\frac{2 p - 3}{p} \overline{t}_\beta^{2 p - 3} \norm{v_0}_{L^p}^p.
\end{multline*}
It follows from the inclusion $v_0 \in \NP_0$ that
\begin{multline}
\label{eqn:aux:1}
\frac{1}{2} \parens*{1 - \frac{1}{\bar{t}_\beta^2}} \norm{v_0}_{L^2}^2
+
\\
+
\frac{1}{4}
\int \int
	\parens*{
		\frac{3}{\abs{x - y}} + \frac{1}{\overline{t}_\beta \beta}
	}
	e^{- \abs{x - y} / \parens{\overline{t}_\beta \beta}}
	v_0 \parens{x}^2 v_0 \parens{y}^2
\dif x \dif y
=
\\
=
\frac{2 p - 3}{p}
\parens{1 - \overline{t}_\beta^{2 p - 6}}
\norm{v_0}_{L^p}^p.
\end{multline}

Let us show that
\begin{equation}
\label{eqn:integral-to-zero}
\int \int
	\parens*{
		\frac{3}{\abs{x - y}} + \frac{1}{\overline{t}_\beta \beta}
	}
	e^{- \abs{x - y} / \parens{\overline{t}_\beta \beta}}
	v_0 \parens{x}^2 v_0 \parens{y}^2
\dif x \dif y
\xrightarrow[\beta \to 0]{}
0.
\end{equation}
It suffices to prove that if $0 < \beta_n \to 0$ as $n \to \infty$, then, up to subsequence,
\[
\int \int
	\parens*{
		\frac{3}{\abs{x - y}} + \frac{1}{\overline{t}_{\beta_n} \beta_n}
	}
	e^{- \abs{x - y} / \parens{\overline{t}_{\beta_n} \beta_n}}
	v_0 \parens{x}^2 v_0 \parens{y}^2
\dif x \dif y
\xrightarrow[n \to \infty]{}
0.
\]
As $\lim_{n \to \infty} \overline{t}_{\beta_n} \beta_n = 0$, then, up to subsequence, $\parens{\bar{t}_{\beta_n} \beta_n}_{n \in \nat}$ is decreasing, so the limit follows from the Monotone Convergence Theorem.

We claim that $\lim_{\beta \to 0} \bar{t}_\beta = 1$. By contradiction, suppose that $0 < \beta_n \to 0$ as $n \to \infty$ and
$\alpha := \liminf_{n \to \infty} \bar{t}_{\beta_n} < 1$. In view of \eqref{eqn:aux:1} and \eqref{eqn:integral-to-zero}, it follows that
\[
0
>
\frac{1}{2} \parens*{
	1
	-
	\frac{1}{\alpha^2}
}
\norm{v_0}_{L^2}^2
=
\frac{2 p - 3}{p}
\parens*{1 - \limsup_{n \to \infty} \overline{t}_\beta^{2 p - 6}}
\norm{v_0}_{L^p}^p
\geq
0,
\]
which is absurd, hence the result.

In view of \eqref{eqn:integral-to-zero}, the limit
$\bar{t}_\beta \to 1$ as $\beta \to 0$ implies
\begin{multline*}
m_\beta
\leq
\I_\beta \parens*{\overline{t}_\beta^2 v_0 \parens{\overline{t}_\beta \cdot}}
=
\frac{p - 3}{2 p - 3} \overline{t}_\beta^3 \norm{v_0}_{D^{1, 2}}^2
+
\frac{p - 2}{2 p - 3} \overline{t}_\beta \norm{v_0}_{L^2}^2
+
\\
+
\frac{p - 3}{2 \parens{2 p - 3}}
\overline{t}_\beta^3
\int
	\parens{v_0^2 \ast \K_{\overline{t}_\beta \beta}}
	\parens{x} v_0 \parens{x}^2
\dif x
+
\\
+
\frac{\overline{t}_\beta^2}{4 \parens{2 p - 3} \beta}
\int \int
	e^{- \abs{x - y} / \parens{\overline{t}_\beta \beta}}
	v_0 \parens{x}^2 v_0 \parens{y}^2
\dif x \dif y
\xrightarrow[\beta \to 0]{}
\I_0 \parens{v_0} = m_0,
\end{multline*}
and the lemma is proved.
\end{proof}

We can use the previous lemma to control the $H^1$-norm of least energy solutions to \eqref{eqn:SBP} for sufficiently small
$\beta$.

\begin{lem}
\label{lem:upper-bound}
$\limsup_{\beta \to 0} \norm{v_\beta}_{H^1} < \infty$.
\end{lem}
\begin{proof}
As $v_\beta \in \NP_\beta$, we obtain
\begin{multline*}
m_\beta
=
\I_\beta \parens{v_\beta}
=
\frac{p - 3}{2 p - 3}
\norm{v_\beta}_{D^{1, 2}}^2
+
\frac{p - 2}{2 p - 3}
\norm{v_\beta}_{L^2}^2
+
\\
+
\frac{p - 3}{2 \parens{2 p - 3}}
\int \int
	\K_\beta \parens{x - y} v_\beta \parens{x}^2 v_\beta \parens{y}^2
\dif x \dif y
+
\\
+
\frac{1}{4 \parens{2 p - 3}}
\int \int
	\frac{e^{- \abs{x - y} / \beta}}{\beta}
	v_\beta \parens{x}^2 v_\beta \parens{y}^2
\dif x \dif y,
\end{multline*}
so
$
m_\beta \geq \parens{p - 3} \norm{v_\beta}_{H^1}^2 / \parens{2 p - 3}
$
and the result follows from Lemma \ref{lem:m0_geq_mbeta}.
\end{proof}

The following inequality follows from an application of Hölder's and Young's inequalities.
\begin{lem}
\label{lem:inequality}
Given $w \in L^4$, it holds that
\[
\int \int
	\parens*{
		\frac{3}{\abs{x - y}} + \frac{1}{\beta}
	}
	e^{- \abs{x - y} / \beta} w \parens{x}^2 w \parens{y}^2
\dif x \dif y
\leq
20 \pi \beta^2 \norm{w}_{L^4}^4.
\]
It follows that if $\set{w_\beta}_{\beta > 0} \subset H^1$ is such that $\limsup_{\beta \to 0} \norm{w_\beta}_{H^1} < \infty$, then
\[
\int \int
	\parens*{
		\frac{3}{\abs{x - y}} + \frac{1}{\beta}
	}
	e^{- \abs{x - y} / \beta}
	w_\beta \parens{x}^2 w_\beta \parens{y}^2
\dif x \dif y
\xrightarrow[\beta \to 0]{}
0.
\]
\end{lem}

Let us show that the family of Nehari--Pohožaev manifolds
$\parens{\NP_\beta}_{\beta > 0}$ is bounded away from zero in $L^p$.

\begin{lem}
\label{lem:lower-bound}
$\inf_{\beta > 0} \set{\norm{v}_{L^p} : v \in \NP_\beta} > 0$.
\end{lem}
\begin{proof}
We claim that
\begin{equation}
\label{eqn:aux:3}
\inf_{\beta > 0} \set*{\norm{v}_{H^1} : v \in \NP_\beta} > 0.
\end{equation}
Indeed, the elementary inequality $r e^{-r} \leq 1 - e^{- r}$ for every $r \geq 0$ implies
\begin{equation}
\label{eqn:aux:2}
0
=
\Pee_\beta \parens{v}
\geq
\frac{1}{2} \norm{v}_{H^1}^2 - \frac{2 p - 3}{p} \norm{v}_{L^p}^p,
\end{equation}
and thus $\parens{2 p - 3} c \norm{v}_{H^1}^{p - 2} / p \geq 1 / 2$, where $c > 0$ denotes the constant of the Sobolev embedding $H^1 \hookrightarrow L^p$.

In this situation, the lemma follows from \eqref{eqn:aux:3} and \eqref{eqn:aux:2}.
\end{proof}

The inclusion $v_\beta \in \NP_\beta$ implies
\[
\Pee_0 \parens{v_\beta}
=
\frac{1}{4}
\int \int
	\parens*{
		\frac{3}{\abs{x - y}}
		+
		\frac{1}{\beta}
	}
	e^{- \abs{x - y} / \beta}
	v_\beta \parens{x}^2 v_\beta \parens{y}^2
\dif x \dif y
>
0,
\]
so there exists a unique $t_\beta > 1$ such that
$t_\beta^2 v_\beta \parens{t_\beta \cdot} \in \NP_0$, i.e.,
\begin{multline}
\label{eqn:v_beta-in-NP_0}
\frac{3}{2} t_\beta^3 \norm{v_\beta}_{D^{1, 2}}^2
+
\frac{1}{2} t_\beta \norm{v_\beta}_{L^2}^2
+
\frac{3}{4} t_\beta^3
\int \int
	\frac{v_\beta \parens{x}^2 v_\beta \parens{y}^2}{\abs{x - y}}
\dif x \dif y
=
\\
=
\frac{2 p - 3}{p} t_\beta^{2 p - 3} \norm{v_\beta}_{L^p}^p.
\end{multline}
Our last preliminary result shows that $t_\beta \to 1$ as
$\beta \to 0$.

\begin{lem}
\label{lem:t_beta-to-1}
$t_\beta \to 1$ and
$\I_0 \parens{t_\beta^2 v_\beta \parens{t_\beta \cdot}} \to m_0$
as $\beta \to 0$.
\end{lem}
\begin{proof}
Let us prove that $t_\beta \to 1$ as $\beta \to 0$. We only have to show that $\limsup_{\beta \to 0} t_\beta \leq 1$. By contradiction, suppose that $\limsup_{\beta \to 0} t_\beta > 1$. In particular, we can fix $\set{\beta_n}_{n \in \nat} \subset \ooi{0, \infty}$ such that
$\beta_n \to 0$ as $n \to \infty$ and
$\alpha := \liminf_{n \to \infty} t_{\beta_n} > 1$. It follows from \eqref{eqn:v_beta-in-NP_0} and from the fact that
$v_{\beta_n} \in \NP_{\beta_n}$ that
\begin{multline*}
\frac{1}{2}
\parens*{\frac{1}{t_{\beta_n}^2} - 1}
\norm{v_{\beta_n}}_{L^2}^2
+
\\
+
\frac{1}{4}
\int \int
	\parens*{
		\frac{3}{\abs{x - y}} + \frac{1}{{\beta_n}}
	}
	e^{- \abs{x - y} / {\beta_n}}
	v_{\beta_n} \parens{x}^2 v_{\beta_n} \parens{y}^2
\dif x \dif y
=
\\
=
\frac{2 p - 3}{p}
\parens{t_{\beta_n}^{2 p - 6} - 1}
\norm{v_{\beta_n}}_{L^p}^p.
\end{multline*}
In view of Lemmas \ref{lem:upper-bound}--\ref{lem:lower-bound},
\begin{multline*}
0
\geq
\frac{1}{2}
\parens*{\frac{1}{\alpha^2} - 1}
\parens*{\limsup_{n \to \infty} \norm{v_{\beta_n}}_{L^2}^2}
\geq
\\
\geq
\frac{2 p - 3}{p}
\parens{\alpha^{2 p - 6} - 1}
\parens*{
	\liminf_{n \to \infty} \norm{v_{\beta_n}}_{L^p}^p
}
>
0,
\end{multline*}
which is absurd, hence the result.

Now, we want to show that
$\I_0 \parens{t_\beta^2 v_\beta \parens{t_\beta \cdot}} \to m_0$ as
$\beta \to 0$. We have
\begin{multline*}
\label{eqn:m_0-leq-m_beta}
m_0
\leq
\I_0 \parens*{t_\beta^2 v_\beta \parens{t_\beta \cdot}}
=
\frac{p - 3}{2 p - 3}
t_\beta^3
\norm{v_\beta}_{D^{1, 2}}^2
+
\frac{p - 2}{2 p - 3}
t_\beta
\norm{v_\beta}_{L^2}^2
+
\\
+
\frac{p - 3}{2 \parens{2 p - 3}} t_\beta^3
\int
	\frac{v_\beta \parens{x}^2 v_\beta \parens{y}^2}{\abs{x - y}}
\dif x
=
t_\beta^3 m_\beta
+
\frac{p - 2}{2 p - 3}
\parens{t_\beta - t_\beta^3}
\norm{v_\beta}_{L^2}^2
+
\\
+
\frac{p - 3}{2 \parens{2 p - 3}}
t_\beta^3
\int \int
	\frac{e^{- \abs{x - y} / \beta}}{\abs{x - y}}
	v_\beta \parens{x}^2 v_\beta \parens{y}^2
\dif x \dif y
-
\\
-
\frac{1}{4 \parens{2 p - 3}}
\overline{t}_\beta^3
\int \int
	e^{- \abs{x - y} / \beta}
	v_\beta \parens{x}^2 v_\beta \parens{y}^2
\dif x \dif y.
\end{multline*}
Due to the limit $\lim_{\beta \to 0} t_\beta = 1$, the result follows from Lemmas \ref{lem:m0_geq_mbeta}--\ref{lem:inequality}.
\end{proof}

Even though this limit will not be explicitly used to prove the theorem, we remark that Lemma \ref{lem:t_beta-to-1} implies
$m_\beta \to m_0$ as $\beta \to 0$ because
$\I_\beta \parens{v_\beta} = m_\beta$ by definition. Let us finally prove the theorem.

\begin{proof}[Proof of Theorem \ref{thm:asymptotic_profile}]
Due to Lemma \ref{lem:t_beta-to-1},
$
\parens{
	u_n := t_{\beta_n}^2 v_{\beta_n} \parens{t_{\beta_n} \cdot}
}_{n \in \nat}
$
is a minimizing sequence of $\I_0|_{\NP_0}$. Let $\mu_n$ denote the measure defined in Lemma \ref{lem:compactness} and let
$\set{\xi_n}_{n \in \nat} \subset \real^3$ be furnished by the same lemma. It follows from Lemmas \ref{lem:upper-bound} and \ref{lem:t_beta-to-1} that
$\set{w_n := u_n \parens{\cdot - \xi_n}}_{n \in \nat}$ is bounded in $H^1$, so there exists $\overline{v}_0 \in H^1$ such that, up to subsequence, $w_n \rightharpoonup \overline{v}_0$ in $H^1$ as
$n \to \infty$. Due to the Kondrakov Theorem, we can suppose further that $w_n \to \overline{v}_0$ a.e. as $n \to \infty$.

Now, we argue as in \cite[Proof of Theorem 1.1]{azzolliniGroundStateSolutions2008} to prove that
\begin{equation}
\label{eqn:convergence-in-L^q}
\norm{w_n - \overline{v}_0}_{L^q} \xrightarrow[n \to \infty]{L^q} 0
\quad \text{for every} \quad
q \in \coi{2, 6}.
\end{equation}
Due to Lemma \ref{lem:compactness},
$
\norm{w_n}_{H^1 \parens{\real^3 \setminus B_{r_\delta} \parens{0}}}^2
<
\delta
$
for every $n \in \nat$. Consider a fixed $\delta > 0$. Due to the Kondrakov Theorem and the fact that $\norm{\cdot}_{H^1}$ is weakly lower-semicontinuous, we obtain
\begin{align}
\label{eqn:ineq-delta}
\norm{w_n - \overline{v}_0}_{L^q}
&\leq
\norm{w_n - \overline{v}_0}_{L^q \parens*{B_{r_\delta} \parens{0}}}
+
\norm{w_n - \overline{v}_0}_{
	L^q \parens*{\real^3 \setminus B_{r_\delta} \parens{0}}
};\nonumber
\\
&\leq
\delta
+
C
\parens*{
	\norm{w_n}_{
		H^1 \parens*{\real^3 \setminus B_{r_\delta} \parens{0}}
	}
	+
	\norm{\overline{v}_0}_{
		H^1 \parens*{\real^3 \setminus B_{r_\delta} \parens{0}}
	}
};\nonumber
\\
&
\leq 3 \delta
\end{align}
\sloppy
for sufficiently large $n \in \nat$, where $C > 0$ denotes the constant of the Sobolev embedding $H^1 \hookrightarrow L^q$. The result then follows from the fact that given $\delta > 0$, there exists $n_\delta \in \nat$ such that \eqref{eqn:ineq-delta} holds for
$n \geq n_\delta$.

We claim that
\begin{multline}
\label{eqn:convergence-two}
\frac{p - 3}{2 p - 3} \norm{w_n}_{D^{1, 2}}^2
+
\frac{p - 2}{2 p - 3} \norm{w_n}_{L^2}^2
\xrightarrow[n \to \infty]{}
\\
\xrightarrow[n \to \infty]{}
\frac{p - 3}{2 p - 3} \norm{\overline{v}_0}_{D^{1, 2}}^2
+
\frac{p - 2}{2 p - 3} \norm{\overline{v}_0}_{L^2}^2.
\end{multline}
Indeed, in view of \eqref{eqn:convergence-in-L^q} and Lemma \ref{lem:NP_0-is-natural-constraint}, we deduce that
$\norm{\overline{v}_0}_{L^p} > 0$, so $\overline{v}_0 \not \equiv 0$.
Considering \eqref{eqn:convergence-in-L^q}, Lemmas \ref{lem:BL}, \ref{lem:t_beta-to-1} and the fact that $\norm{\cdot}_{D^{1, 2}}$ is weakly lower-semicontinuous, we obtain
$
\Pee_0 \parens{\overline{v}_0}
\leq
\liminf_{n \to \infty} \Pee_0 \parens{w_n}
=
0
$.
As $\overline{v}_0 \not \equiv 0$, we deduce that there exists a unique $t_0 \in \oci{0, 1}$ such that
$t_0^2 \overline{v}_0 \parens{t_0 \cdot} \in \NP_0$. We obtain
\begin{multline*}
m_0
\leq
\I_0 \parens*{t_0^2 \overline{v}_0 \parens{t_0 \cdot}}
=
\\
=
\frac{p - 3}{2 p - 3} t_0^3 \norm{\overline{v}_0}_{D^{1, 2}}^2
+
\frac{p - 2}{2 p - 3} t_0 \norm{\overline{v}_0}_{L^2}^2
+
\frac{p - 2}{2 \parens{2 p - 3}} t_0^3
\int
	\frac{
		\overline{v}_0 \parens{x}^2 \overline{v}_0 \parens{y}^2
	}{\abs{x - y}}
\dif y
\dif x
\leq
\\
\leq
\frac{p - 3}{2 p - 3} \norm{\overline{v}_0}_{D^{1, 2}}^2
+
\frac{p - 2}{2 p - 3} \norm{\overline{v}_0}_{L^2}^2
+
\frac{p - 2}{2 \parens{2 p - 3}}
\int
	\frac{
		\overline{v}_0 \parens{x}^2 \overline{v}_0 \parens{y}^2
	}{\abs{x - y}}
\dif y
\dif x
\leq
\\
\leq
\I_0 \parens{w_n} + o_n \parens{1}
\end{multline*}
for sufficiently large $n \in \nat$ and the result follows by taking the limit $n \to \infty$.

In view of \eqref{eqn:convergence-in-L^q} and \eqref{eqn:convergence-two}, we obtain $\norm{w_n - \overline{v}_0}_{H^1} \to 0$ as
$n \to \infty$, so $\overline{v}_0 \in \NP_0$ and
$\I_0 \parens{\overline{v}_0} = m_0$. Finally, the fact that
$\I_0' \parens{\overline{v}_0} = 0$ is a corollary of Lemma \ref{lem:NP_0-is-natural-constraint}.
\end{proof}

\sloppy
\printbibliography
\end{document}